\documentclass{amsart}

\usepackage{amsthm,amssymb,latexsym,amsmath}
\usepackage[all]{xypic}
\usepackage{dsfont}
\usepackage{color}

\usepackage[symbol]{footmisc}

\newcommand{\catA}[1]{{\mathfrak A}}
\newcommand{\catI}[1]{{\mathfrak I}}
\newcommand{\catS}[1]{{\mathfrak S}}

\newcommand{\p}[1]{{\mathbb{P}^{#1}}}

\newcommand{\pn}{{\mathbb{P}^n}}

\newcommand{\pnu}{\mathbb{P}^{n-1}}
\newcommand{\pnd}{\mathbb{P}^{n-2}}

\newcommand{\pe}{\mathbb{P}(\mathcal{E})}

\newcommand{\bp}[1]{\widetilde{\mathbb{P}^{#1}}}

\newcommand{\bpn}{{\widetilde{\mathbb{P}^n}}}

\DeclareMathOperator{\Det}{{\rm det}}

\DeclareMathOperator{\Homc}{\mathcal{H}\emph{om}}

\DeclareMathOperator{\ho}{{\rm H}}

\newtheorem{theorem}{Theorem}[section]

\newtheorem{proposition}[theorem]{Proposition}
\newtheorem{lemma}[theorem]{Lemma}
\newtheorem{corollary}[theorem]{Corollary}

\newtheorem{remark}[theorem]{Remark}

\newtheorem{definition}[theorem]{{\bf Definition}}

\begin{document}

\title{Torsion free instanton sheaves on the blow-up of $\mathbb{P}^{n}$ at a point}

\author{Abdelmoubine Amar Henni$^{\dagger}$}

\date{}

\maketitle

\vspace{1cm}

\begin{abstract}
We define the analogue of instanton sheaves on the blow-up $\widetilde{\mathbb{P}^n}$ of the $n$-dimensional projective space at a point. We choose an appropriate polarisation on $\widetilde{\mathbb{P}^n}$ and construct rank $2$ examples of locally free and non locally free (but torsion free) type. In general, the defined instantons also turn out to be the cohomology of monads, although non-linear ones. Moreover, in the five dimensional case, we show that there are continuous families of them that fill, at least, a smooth component in the moduli of semi-stable sheaves.

\end{abstract}

\vspace{1cm}

\tableofcontents

\vspace{1cm}

\section{Introduction}
\label{IN}

\footnotetext[2]{The author was partially supported by the CAPES-COFECUB $08/ 2018$ project and MATH-AMSUD project: GS$\&$MS $21$-MATH$-06$  $2021\hspace{0.1cm}\& \hspace{0.1cm}2022.$
}

Let $\bpn$ be the blow-up of the projective $n$-dimensional space $\pn,$ over the complex numbers, at a point. In this paper we study analogues of instanton sheaves on $\bpn,$ as an example of a higher dimensional ($n\geq3$) generalization of instantons on Fano varieties. On surfaces and threefolds there are many examples of analogues of instantons, see \cite{rava, Henni0, kuznetsov, faenzi, ccgm,AM,AM1,Sanna, MMJ} and references therein. But for $n>3,$ there is little done on geometries that are not $\pn,$ \cite{J1,FJM,HJM}, such as varieties with poly-cyclic Picard group, i. e., when ${\rm Pic}(X)\simeq\mathbb{Z}^{\oplus k}$ for $k>1.$ In the case under study, $\bpn$ is Fano with Picard number $2,$ we fix a specific polarisation $\mathcal{L}_{n}$ in each dimension, but we will drop the dimension index, for now, in order to simplify the notation, and will simply write $\mathcal{L}.$ We shall introduce in Definition \ref{Inst-def-bpn} the notion of an instanton sheaf on $\bpn$, as a $\mu_{\mathcal{L}}$-semi-stable torsion free sheaf $\mathcal{E}$ satisfying a set of cohomological conditions, and having first Chern class according to the formula $\Det(\mathcal{E}):=\mathcal{L}^{\otimes 2}\otimes\omega_{\bpn},$
where $\omega_{\bpn}$ is the canonical sheaf of $\bpn.$ In this regard, the polarisation is useful to fix a given notion of stability and intervenes in fixing the first Chern class but has no role in the cohomological conditions, so far.

A wider definition might include cohomological conditions depending on the chosen polarisation, but finding an explicit example, at least in the context of this preliminary study, is challenging and becomes difficult with the change of polarisation. The examples we give below are constructed by using the Hartshorne-Serre correspondence and for our choice of polarisation, as above, we shall see that there are explicit examples of instanton analogues in the odd dimensional case. The other general construction method of bundles that one can try, in the future, is given by pushing forward line bundles over some branched covers over the base variety. 

As for the first Chern class formula with respect to $\mathcal{L},$ it can be understood as a constraint for which $\ho^{i}(\mathcal{E}\otimes\mathcal{L}^{-1})$ is Serre dual to $\ho^{n-i}(\mathcal{E}\otimes\mathcal{L}^{-1}),$ thus reducing the instantonic cohomological conditions to just $\ho^{1}(\mathcal{E}\otimes\mathcal{L}^{-1})=0,$ so that a whole column $\ho^{i}(\mathcal{E}\otimes\mathcal{L}^{-1}),$ in the cohomology table $\oplus_{k\in\mathbb{Z}}\ho^{i}(\mathcal{E}\otimes\mathcal{L}^{k})$ of $\mathcal{E},$ is trivial. This is similar to what happens for the classical mathematical instanton on $\p3,$ when $\mathcal{L}=\mathcal{O}_{\p3}(2),$ and in a proposed definition of rank $2$ instantons, in \cite{AM}, on projective threefolds.

The delineation of instantons that we propose (Definition \ref{Inst-def-bpn}) also leads to a description in terms of {\em monads} (Lemma \eqref{Gen-monad}), that is $3$ term complexes of vector bundles with only non trivial cohomology in the middle \cite[II, \S 3]{OSS}. Although the monads we obtain are non linear. Linear monads are those in which each term in the complex is a sum of copies of the structure sheaf, the line bundle associated to a hyperplane and its dual. It is very difficult to reduce the algebraic data in morphism of non linear monads, in order to produce ADHM type data, which, in its own right, can be very useful, for instance, in constructing examples of instantons of higher rank.  

For rank two, we exhibit a prototype $\mathcal{E}$ of an instanton on $\bp5,$ of charge $c:=c_{2}\cdot(\mathcal{L})^{3}=8.$ Its occurrence is due to the Hartshorne-Serre correspondence \cite{arrondo}, and for its $\mu_{\mathcal{L}}$-semi-stability, we use a generalized Hoppe criteria on poly-cyclic varieties \cite{HJ,JMPS}, which is similar to the technique used in \cite{MMJ, ccgm}. We also study its restriction on a divisor $D$ which is either the pull back $H$ of a hyperplane in $\p5,$ via the blow-down map or the exceptional divisor $E.$ In Corollary \ref{div-restriction}, we show that $$\mathcal{E}|_{D}\cong\left\{ \begin{array}{cc}\mathcal{O}_{\p4}(-1)^{\oplus 2} & \textnormal{ if } D=H \\ \mathcal{O}_{\p4}\oplus\mathcal{O}_{\p4}(-2) & \textnormal{ if }D=E \end{array}\right.$$ This is specific to our situation, and we don't expect that it holds for any instanton on $\bpn,$ at least at this preliminary stage. Moreover, we give the explicit monad describing $\mathcal{E},$ in this case (Corrolary \ref{monad-p5}).

Furthermore, the $\bp5$ prototype is generalized, in Section \ref{higher-dim}, for all odd dimension $n\geq5,$ leading to instantons of charge $8$ in every such dimension. We do not know if there are instantons of lower charges neither what would be the minimum.

In this general setting, we also provide examples of non locally free instantons obtained via {\em elementary transformations}. It is interesting to know whether these non locally free instantons are smoothable so that they belong to components of higher charges, as it has been done in the case of $\bp3$ \cite{Henni1}. We remark that the definition given in \cite{Henni1} for the torsion free instantons also follows from our present proposal.  

One can try to adapt the same technical strategy in order to provide examples of instantons on even spaces. However, in all examples we have built so far, the cohomological conditions of our definition of instanton are satisfied but the semi-stability fails, as will be shown in the last section of the paper. This might be addressed in the future, by seeking other variations of the construction technique.

\vspace{0.5cm}

We now summarize how the results are organized in this paper; in Section \ref{blow-up} we remind briefly, the readers, of some preliminaries on the geometry of the blow-ups of projective spaces at a point that we shall use, and also introduce some other general results that are needed in subsequent sections.

In Section \ref{inst-definitions} we review several recent definitions of instanton sheaves over various geometries that mimic the behaviour of classical mathematical instantons on $\p3,$ particularly recent definitions of instantons on Fano $3$-folds \cite{kuznetsov, faenzi, ccgm, AM}, some of which are related to Ulrich bundles \cite{costa-miro,ccgm,gaia} and bundles with natural cohomology \cite{floystad,J1}.

In Section \ref{odd}, we define the analogues of instanton sheaves on blow-ups $\bpn$ of projective spaces at a point. We show how they can be described as cohomologies of (non-linear) monads in odd dimensions. As a prototype, in Section \ref{5dim}, we construct an explicit rank $2$ example on $\bp5$ by using Hartshorne-Serre correspondence (Proposition \ref{properties}, Theorem \ref{instanton-proof}) and prove that there is at least a smooth component of dimension $5$ or $6,$ inside the moduli of semi-stable sheaves, whose members are instantons (Proposition \ref{component}). Moreover, the construction on $\bp5$ is also shown to work for odd dimensions $n\geq5$ (Theorem \ref{Gen-const-Inst}). As a result, using {\em elementary transformations}, we obtain torsion-free (but non-locally-free) examples of instantons in Section  \ref{non-loc-free}.

Finally, in Section \ref{even}, we discuss the even dimensional case and why the examples obtained by a similar construction, as in the odd dimensional case, lead to bundles satisfying all conditions of our definition except semi-stability. This might be an artifact of the specific examples or construction we are using, and this might be addressed in the future.

\vspace{0.1cm}

\section{The blow-up of projective spaces at a point}\label{blow-up}

Let $\pn$ be the projective $n$-space, over the field of complex numbers $\mathbb{C},$ and fix a point $p_{0}$ in it. We will denote by $\bpn$ be the blow-up of $\pn$ at $p_{0},$ and by $\pi:\bpn\to\pn$ the blow-down map. It is well known that $\bpn$ has the structure of the projective bundle $\pe:=\mathbb{P}(\mathcal{O}_{\pnu}(-1)\oplus\mathcal{O}_{\pnu})$ on $\pnu.$ We will denote the canonical projection by $pr:\bpn\to\pnu$. It follows that the Chow ring $A^{\ast}(\bpn)$ of $\bpn$ is given by $$\frac{\mathbb{Z}[\alpha, \xi]}{(\alpha^{n},\xi^{2}-\alpha\xi)},$$
where $\alpha$ is the pull-back of a hyperplane class in $\pnu$ and $\xi=c_{1}(\mathcal{O}_{\pe}(1)),$ which is also identified with the pull-back of hyperplane class in $\pn.$ In particular the Picard group is ${\rm Pic}(\bpn)\cong\xi\cdot\mathbb{Z}\oplus\alpha \cdot\mathbb{Z}$ and the class of the exceptional divisor $E$ is given by $\xi-\alpha.$
Explicitly, one can use the following expansion for an element $c^{\ast}$ of the Chow ring $A^{\ast}(\bpn)$
\begin{equation}
c^{\ast}=r_{0}\cdot1 \oplus (\oplus_{k=1}^{n}r_{k}\xi^{k})\oplus (\oplus_{l=1}^{n-1}s_{i}\alpha^{l}),
\end{equation}
where the coefficients $r_{k},$ $k=0,\cdots n,$ and $s_{l},$ $l=1,\cdots n-1,$ are integers, and the product is subject to the relations $\xi^{k}\cdot\alpha^{l}=\xi^{k+l},$ for $k+l\leq n,$ $k\leq n-1$ and zero otherwise. These relations follow from the conditions $\alpha^{n}=0$ and $\xi^{2}=\alpha\cdot\xi.$ 

In the sequel we shall denote the line bundles $\mathcal{O}_{\bpn}(p\xi+q\alpha),$ associated to a divisor $p\xi+q\alpha,$ by $\mathcal{O}_{\bpn}(p,q),$ unless we explicitly need to specify the generators. 

Moreover, one can use the restriction sequence $$0\to T_{\bpn/\pnu}\to T\bpn\to pr^{\ast}T\pnu\to0$$ and the tautological sequence $$0\to\mathcal{O}_{\pe}(-1)\to pr^{\ast}\mathcal{E}\to Q\to0,$$ where $Q$ is the universal quotient, in order to compute the canonical sheaf $\omega_{\bpn}$ of $\bpn.$ By taking determinants in the sequences above and by using the fact that $$T_{\bpn/\pnu}\cong \Homc(\mathcal{O}_{\pe}(-1),Q)\cong\mathcal{O}_{\pe}(1)\otimes Q,$$ this gives $$\omega_{\bpn}=\mathcal{O}_{\bpn}(-2,1-n).$$

In order to compute the dimensions $h^{i}$ for the cohomology spaces of line bundles on $\bpn$ one can use the formulae for the dimensions of line bundles over projective spaces \cite[Ch I,\S1.1]{OSS} and the fact that for $p\geq0$ $^{\dagger\dagger}$
\footnotetext[8]{The case $p<0$ is treated similarly by taking the dual $\mathcal{E}^{\vee},$ instead of $\mathcal{E}.$}
\begin{align}
    \ho^{i}(\bpn, \mathcal{O}_{\bpn}(p,q))&=\ho^{i}(\bpn,pr^{\ast}S^{p}(\mathcal{E^{\vee}})\otimes pr^{\ast}\mathcal{O}_{\pnu}(q)) \notag \\
    &=\oplus_{k=0}^{p}\ho^{i}(\pnu, \mathcal{O}_{\pnu}(q+k)), \notag \\
\end{align}
where $S^{p}(\mathcal{E})$ denotes the $p$-th symmetric power of the bundle $\mathcal{E}.$ These dimensions are given by 

\begin{equation}
    \dim\ho^{i}(\bpn, \mathcal{O}_{\bpn}(p,q))=\left\{\begin{array}{ll}
        \sum_{k=0}^{p}\binom{n+q+k-1}{n-1} &  i=0, \\
        \sum_{k=0}^{-2-p}\binom{n+q-k-2}{n-1}  &  i=1 \qquad\qquad  p\geq0\\
        \sum_{k=0}^{p}\binom{-q-k-1}{n-1}  &  i=n-1 \\
        \sum_{k=0}^{-2-p}\binom{k-q}{n-1}  &  i=n\\
        0 & \textnormal{otherwise}
    \end{array} \right. \label{hi-lib}
\end{equation}
where the binomial coefficients are set to zero when the upper argument is smaller than the lower one. Similar formulae are obtained when $p<0.$ In particular $$\ho^{i}(\bpn,\mathcal{O}_{\bpn}(p,q))=0$$ for all $2\leq i\leq n-2$ and all $p,q\in\mathbb{Z},$ since the base space $\pnu$ of the projection $pr$ is aCM (arithmetically Cohen-Macaulay).

\begin{corollary}
$\ho^{i}(\bpn,\mathcal{O}_{\bpn}(p,q))=0,$ for all $i=0,\cdots,n$ and $(p,q)=(0,0), \cdots,\newline (0,1-n)$ or $(p,q)=(-1,1),\cdots,(-1,2-n).$
In particular the following sequence is an exceptional collection of line bundles
$$\mathcal{O}_{\bpn}(-1,2-n),\cdots,\mathcal{O}_{\bpn}(-1,1), \mathcal{O}_{\bpn}(0,1-n),\cdots,\mathcal{O}_{\bpn}.$$
\end{corollary}

\begin{proof}
The proof is a direct application of the formula \eqref{hi-lib}.
\end{proof}



In \cite[Theorem 8]{AO}, Ancona and Ottaviani showed that, for $\mathcal{E}:=\bigoplus_{i} \mathcal{O}_{\pn}(a_{i})$ on $\pn$ with $a_{i}\geq0,$ and for $X=\mathbb{P}(\mathcal{E})\stackrel{pr}{\to}\pn,$ let $\mathcal{U}_{rel},$ $Q$ be, respectively, the relative universal and quotient sheaf. Then any coherent sheaf $\mathcal{F},$ there is a Beilinson-like spectral sequence \cite[Theorem 1]{AO} of $E_2$-term {\small $$E_2^{si}=\bigoplus_{q+h=i}\ho^{s}(pr^{\ast}\mathcal{O}_{\pn}(-q)\otimes\mathcal{U}_{rel}(-h)\otimes\mathcal{F})\otimes pr^{\ast}\Omega^{q}_{\pn}(q)\otimes\Lambda^{h}Q^{\vee},$$} whose total complex is given by $\mathcal{C}^{\bullet}$ and which converges to the original sheaf $\mathcal{F}$, that is, $E_{\infty}^{p}=\bigoplus_{s+i=p} E^{s,i}$ is zero if $p\neq0$ and $E_{\infty}^{0}$ is isomorphic to the graded object of some filtration of $\mathcal{F}.$ This will be useful for our construction of explicit examples.

The Beilinson-like spectral sequence is also obtained by a special Koszul resolution of the diagonal \cite[Theorem 1]{AO}, and the differentials of the complex $\mathcal{C}^{\bullet}$ are inherited from the differentials of the $E_2$ term, which, themselves, come from the ones in the Koszul resolution of the diagonal. 

In our case$^{\ddagger}$ the quotient bundle is the rank $1$ bundle $Q=\mathcal{O}_{\bpn}(1,0)$ and the relative universal bundle is the line bundle $\mathcal{U}_{rel}(-1)=\mathcal{O}_{\bpn}(-1,1).$ Then, the terms of $\mathcal{C}^{\bullet}_{\mathcal{F}},$ in the above theorem, become:

\footnotetext[3]{This is due to the fact that in \cite[Theorem 8]{AO}, the bundle $\mathcal{E}$ is dual to ours, so one has to do a twist by $\mathcal{O}_{\pn}(1)$ to obtain the universal and quotient bundles in our notation.}

\begin{equation}\label{complex}
\mathcal{C}^{p}_{\mathcal{F}}=\bigoplus_{s-i=p}\bigoplus_{q+h=i}\mathbb{H}^{s}(\mathcal{F}\otimes\mathcal{O}_{\bpn}(-h,h-q))\otimes pr^{\ast}\Omega^{q}_{\pnu}\otimes\mathcal{O}_{\bpn}(-h,q). \end{equation}

It might also be useful to compute the dimensions of the cohomology groups of $\Omega^{l}(p,q),$ where $\Omega^{l}=pr^{\ast}\Omega^{l}_{\pnu}.$ As in the case of line bundles \eqref{hi-lib}, one can use Bott's formulae \cite{bott, OSS}, and the result is

\begin{equation}
    \dim\ho^{i}(\bpn, \Omega^{l}(p,q))=\left\{\begin{array}{ll}
        \sum_{k=0}^{p}\binom{q+k+n-l-1}{n-1-l}\binom{q+k-1}{l} &  i=0, \\
        \sum_{k=0}^{-2-p}\binom{q-k+n-l-2}{n-1-l}\binom{q-k-2}{l}  &  i=1 \\
    1    & 0\leq l=i\leq n-1, k+q=0 \\
        \sum_{k=0}^{p}\binom{-q-k+l}{l} \binom{-q-k-1}{n-1-l} &  i=n-1 \\
        \sum_{k=0}^{-2-p}\binom{k-q+l+1}{l} \binom{k-q}{n-1-l} &  i=n\\
        0 & \textnormal{otherwise}
    \end{array} \right. \label{hi-omega}
\end{equation}
where the binomial coefficients are assumed to be zero when the upper argument is smaller than the lower one.

There is an Euler sequence \cite[\S 4.3]{fulton} on $\bpn$

\begin{equation}\label{euler-seq}
    0\to\mathcal{O}_{\bpn}^{\oplus 2}\to\begin{array}{l}
        \mathcal{O}_{\bpn}^{\oplus n}(0,1)   \\
        \oplus \mathcal{O}_{\bpn}(1,0)  \\
        \oplus \mathcal{O}_{\bpn}(1,-1)
    \end{array}\to{\rm T}\bpn\to0,
\end{equation}
which can be used to compute characteristic classes. The Chern polynomial of the tangent bundle is 
\begin{equation}
    {\rm C}({\rm T}\bpn)=(1+\alpha)^{n}(1+2\xi-\alpha),
\end{equation}
and its Todd class is 

\begin{equation}
    {\rm Td}({\rm T}\bpn)=(\frac{\alpha}{1-e^{-\alpha}})^{n}\cdot(\frac{\xi}{1-e^{-\xi}})\cdot(\frac{\xi-\alpha}{1-e^{-(\xi-\alpha)}}).
\end{equation}

Moreover, the Euler characteristic of line bundles is given by 

\begin{equation}\label{euler-chi}
    \chi(\mathcal{O}_{\bpn}(p,q))=\frac{1}{n!}\lbrack(p+q+1)(p+q+2)\cdots(p+q+n)-q(q+1)\cdots(q+n-1)\rbrack.
\end{equation}

\vspace{0.5cm}

One can define a notion of stability with respect to an ample line bundle $\mathcal{L},$ as in \cite{HJ}, for products of projective spaces and \cite{JMPS}, in general, for a variety $X$ with poly-cyclic Picard group, i.e., ${\rm Pic}(X)\cong\mathbb{Z}^{\oplus \rho}$ for some positive integer $\rho.$ 

As in \cite{JMPS} we denote the degree of a sheaf $\mathcal{F}$ by $\delta_{\mathcal{L}}(\mathcal{F}):=c_{1}(\mathcal{F})\cdot c_{1}(\mathcal{L})^{n-1},$ where $n$ is the dimension of $X.$ Then the relative slope of $\mathcal{F}$ is defined by $\mu_{\mathcal{L}}(\mathcal{F}):=\frac{\delta_{\mathcal{L}}(\mathcal{F})}{rk\mathcal{F}}.$ A sheaf $\mathcal{F}$ is said to be $\mu_{\mathcal{L}}$-stable, if for any proper subsheaf $\mathcal{G},$ with torsion free sheaf quotient $\mathcal{F}/\mathcal{G}$ we have 
$$\mu_{\mathcal{L}}(\mathcal{G})<\mu_{\mathcal{L}}(\mathcal{F}).$$
Similarly, semi-stability is defined by substituting $<$ by $\leq$ in the inequality above. We shall mainly use the following special case of \cite[Theorem 3]{JMPS}:

\begin{corollary}\label{rk2-hoppe}
A rank $2$ sheaf $\mathcal{F},$ on a polarized variety $(X,\mathcal{L})$ with poly-cyclic Picard group, is $\mu_{\mathcal{L}}$-stable (semi-stable) if, and only if, 
$$\ho^{0}(X,\mathcal{F}\otimes\theta)=0$$
for every line bundle $\theta$ such that $\delta_{\mathcal{L}}(\theta)\leq-\mu_{\mathcal{L}}(\mathcal{F})$ (respectively, $\delta_{\mathcal{L}}(\theta)<-\mu_{\mathcal{L}}(\mathcal{F})$)
\end{corollary}

We will also need the Hartshorne-Serre correspondence \cite{arrondo, Hart1, serre} in the next sections. We state it the following form.

\begin{theorem}[\cite{arrondo}]\label{hart-ser}
Let $X$ be a smooth algebraic variety and let $Y$ be a locally complete intersection subscheme of codimension $2$ in $X.$ Let $N$ be the normal bundle of $Y$ in $X$ and let $\mathcal{S}$ be a line bundle on $X$ such that $\ho^{2}(X,\mathcal{S}^{-1})=0.$ Assume that $\wedge^{2}N\otimes\mathcal{S}^{-1}|_{Y}$ has $r-1$ linearly independent global sections $s_{1},\cdots,s_{r-1}$ that generate it. Then there exists a rank $r$ vector bundle $\mathcal{F}$ over $X$ such that
\begin{itemize}
    \item[(i)] $\Det(\mathcal{F})=\mathcal{S},$
    \item[(ii)] $E$ has $r-1$ global section $\alpha_{1},\cdots,\alpha_{r-1}$ whose degeneracy locus is $Y$ and such that $s_{1}\alpha_{1}|_{Y}+\cdots+s_{r-1}\alpha_{r-1}|_{Y}=0.$
\end{itemize}
Moreover, if $\ho^{1}(X,\mathcal{S}^{-1})=0,$ then conditions ${\rm (i)}$ and ${\rm (ii)}$ determine $\mathcal{F}$ up to isomorphism.
\end{theorem}

\vspace{1cm}

\section{Instantons through various definitions}\label{inst-definitions}

Instantons were first discovered in the context of Yang Mills theory on $\mathbb{R}^{4}$ and its point compactification $S^{4}$, as anti-self-dual connections \cite{BPST} and their relation with algebraic geometry was discovered through their classification by the celebrated Atiyah, Drinfeld, Hitchin and Manin (ADHM) linear algebraic data \cite{ADHM}. This resulted from the Atiyah-Ward correspondence between gauge theoretic instantons on $S^{4}$ and some holomorphic bundles (now called instantons) on the twistor space $\p3$ \cite{atiyah}. Since then, they have gained a special interest in both algebraic geometry and mathematical physics. Analogues of instanton bundles were also defined and studied on various varieties. On higher dimensional projective spaces, an analogue of the classical instanton on $\p3$ was introduced in the following way:

\begin{definition}\label{classic-inst}(M. Jardim \cite{J1})
An instanton sheaf on $\pn$ for $(n\geq2)$ is a torsion-free coherent sheaf $E$ on $\pn$ with $c_{1}(E)=0$ satisfying the following cohomological conditions:
\begin{itemize}
    \item[(1)] for $n\geq2,$ $\ho^{0}(\pn,E(-1))=\ho^{n}(\pn,E(-n))=0;$
    \item[(2)] for $n\geq3,$ $\ho^{1}(\pn,E(-2))=\ho^{n-1}(\pn,E(1-n))=0;$
    \item[(3)] for $n\geq4,$ $\ho^{p}(\pn,E(k))=0,$ when $2\leq p\leq n-2$ and $k\in\mathbb{Z}$
\end{itemize}
The integer $c=-\chi(E(-1))$ is called the charge of $E.$
\end{definition}

These sheaves are generalizations of the mathematical instantons, on $\p3,$ to the case of higher rank and dimension, so we shall refer to them as classical instantons. Moreover, their descriptions in terms of monads and $ADHM$ data were studied by several authors, as in \cite{HJM, FJM, AMM} and references therein. It is also well-known that the classical instantons are $\mu$-semistable \cite{J1}.

A decade ago, definitions of instantons were proposed for Fano $3$-folds of Picard number $1$ and index $2$ and $3,$ by Kuznetsov \cite{kuznetsov} and Faenzi \cite{faenzi}. 

Recently, a definition of instanton bundles on $\bp3$ has been given by Casnati et al. \cite{ccgm}.  This led Antonelli and Malaspina to propose a definition of instantons on any polarized $3$-fold \cite{AM}. The latter is reproduced below. 

\begin{definition}\label{instanton-3folds}(Antonelli-Malaspina \cite{AM})
Let $(X,\mathcal{L})$ be a polarized projective variety of dimension $3.$ A rank $2$ locally free sheaf $\mathcal{E}$ on $X$ will be called $
\mathcal{L}$-instanton if it satisfies the following conditions:
\begin{itemize}
    \item[(i)] $\Det(\mathcal{E})=\omega_{X}\otimes\mathcal{L}^{\otimes 2}$
    \item[(ii)]$\mathcal{E}$ is $\mu_{\mathcal{L}}$-semi-stable and $\ho^{0}(X,\mathcal{E})=0$
    \item[(iii)] $\ho^{1}(X,\mathcal{E}\otimes\mathcal{L}^{-1})=0$ 
\end{itemize}
\end{definition}

A detailed analysis for $\mathcal{L}$-instanton bundles on $\p3,$ and more generally for Fano $3$-folds, can be found in \cite{faenzi, kuznetsov,Sanna,AM,AM1}. In dimension $3,$ instanton bundles also seem to be related somehow to Ulrich bundles, at least up to twist and for some values of the charge and rank \cite{gaia, ccgm, costa-miro}. Their historical relevance in algebraic geometry and gauge theory applications is also widely known.

When $X$ is the projective space $\pn$ with $n>3,$ any rank $2$ instanton is split (is a sum of line bundles). This can be seen by applying the splitting criteria of Kumar, Peterson, and Rao \cite{KPR} and the analogue of item {\rm (3)} in Definition \ref{classic-inst}. In general, this splitting does not occur if $X$ is not $\pn,$ but can be useful when studying the restriction of these bundles to a subvariety of $X,$ birational to some projective space, as we shall see in Corollary \ref{div-restriction}.

The case of non-locally free but torsion-free instanton sheaves has been widely studied for surfaces and projective spaces.  For more details on the case of instanton sheaves on the blow-up of $\p3$ at a point, see \cite{Henni1}.

In many instances, the parts of the moduli space representing non-locally free sheaves are interesting to understand their compactification. These compactifications of moduli spaces, although depending on the choice of stability, have various gauge theoretic and mathematical physics applications.


\vspace{0.5cm}


\section{Instantons on $\bpn$ for odd $n$}\label{odd}

Fix an odd integer $n\geq3$ and set $N_{n}=\left\{\begin{array}{cc} 1 & \textnormal{if } n=3\\ \frac{n-3}{2}& \textnormal{if }n>3\end{array}\right.$

\begin{definition}\label{Inst-def-bpn}

On $(\bpn, \mathcal{L}:=\mathcal{O}_{\bpn}(1,N_{n})),$ a $\mu_{\mathcal{L}}$-semi-stable rank $r$ torsion free coherent sheaf $\mathcal{F},$ with first Chern class $$c_{1}(\mathcal{F})=\left\{\begin{array}{cc}-2\alpha & \textnormal{if } n>3 \\ 0 & \textnormal{if } n=3\end{array}\right.,$$ will be called {\em instanton} if it satisfies
\begin{itemize}
    \item[(i)] $\ho^{0}(\bpn,\mathcal{F}(0,-2))=0,$ $\ho^{0}(\bpn,\mathcal{F}(-1,0))=0;$
    \item[(ii)] $\ho^{1}(\bpn,\mathcal{F}(-1,-1))=0,$ $\ho^{n-1}(\bpn,\mathcal{F}(0,3-n))=0;$  
    \item[(iii)] $\ho^{n}(\bpn,\mathcal{F}(0,2-n))=0,$ $\ho^{n}(\bpn,\mathcal{F}(-1,4-n))=0$
    \item[(iv)] for $n\geq4,$ $\ho^{1}(\bpn,\mathcal{F}(0,-3))=0=\ho^{n-1}(\bpn,\mathcal{F}(-1,5-n))$ and also $\ho^{i}(\bpn,\mathcal{E}\otimes\mathcal{L}')=0$, for all $2\leq i\leq n-2$ and all invertible sheaves $\mathcal{L}'$ on $\bpn.$
\end{itemize}

\end{definition}

This definition is a generalization of the one given in  \cite{ccgm} from the case of $\bp3$ to the case of the blowup of a higher-dimensional projective space. In contrast to the case of $\bp3,$ the instantons are defined just for one particular choice of polarization in each odd dimension $n.$ The pattern of generalization relating the definition from \cite{ccgm} to \ref{Inst-def-bpn} is very similar to the one relating the definition of instantons on $\p3$ to \ref{instanton-3folds}.

The chosen polarisation $\mathcal{L},$ on $\bpn$ only affects the semi-stability of the sheaf and the selection of the first class as $\Det(\mathcal{F})=\mathcal{L}^{\otimes2}\otimes\omega_{\bpn}.$ 
As mentioned in the introduction, this can be seen as a constraint in the rank two locally free case, for which $\ho^{i}(\mathcal{E}\otimes\mathcal{L}^{-1})\cong\ho^{n-i}(\mathcal{E}\otimes\mathcal{L}^{-1}),$ by Serre duality, thus reducing the instantonic cohomological conditions to just $\ho^{1}(\mathcal{E}\otimes\mathcal{L}^{-1})=0.$  As a result, for all $i,$ the groups $\ho^{i}(\mathcal{F}\otimes\mathcal{L}^{-1})$ in the cohomology table of $\mathcal{F}$ vanish. This is similar to what happens for the classical mathematical instanton on $\p3$ with $\mathcal{L}=\mathcal{O}_{\p3}(2)$ and in the definition of rank $2$ instantons on projective threefolds proposed in \cite{AM}.

However, the cohomological conditions in our definition seem independent of $\mathcal{L}.$ Therefore, it would be interesting to see what happens when the polarization is changed in order to make the definition work for any polarisation on $\bpn$, and possibly for more general varieties. It is also compelling to know for which set of polarizations, if any, the sheaves that satisfy the cohomological conditions ${\rm (i)-(iv)},$ in the definition above, can be (semi-) stable. This is beyond the scope of this paper, which concentrates on the existence, and some simple properties, of objects satisfying the above definition in the odd dimensional case. 
Nevertheless, it might be one of the next steps to explore. Another direction, which is worth investigating, is the existence of higher rank examples. The techniques of extracting ADHM data from monads, combined with Fl\o ystad's main result \cite{floystad}, yield many examples of higher rank classical instantons (Definition \ref{classic-inst}) on $\pn.$ As we shall see below, the present definition of instantons on $\bpn$ also leads to monads, although not linear ones. Studying the numerical conditions for the existence of the obtained monads and whether they might produce ADHM-like data is a difficult problem, but it may just as well help in the discovery of some new examples on $\bpn,$ particularly of higher rank. The even dimensional case will be discussed in Section \ref{even}.

\begin{lemma}\label{Gen-monad}
An instanton sheaf as defined above (Definition \ref{Inst-def-bpn}) is the cohomology of a monad $$\mathbb{M}:\quad0\to\mathcal{C}^{-1}\to\mathcal{C}^{0}\to\mathcal{C}^{1}\to0,$$ where
\begin{equation}
\mathcal{C}^{-1}=\begin{array}{c}
   \ho^{1}(\mathcal{F})\otimes\mathcal{O}_{\bpn} \quad \oplus \quad \ho^{n-1}(\mathcal{F}(0,2-n))\otimes\Omega^{1}(0,n-2) \\
   \oplus\ho^{n-1}(\mathcal{F}(-1,4-n))\otimes\Omega^{n-3}(-1,n-3) \quad \oplus \quad \ho^{n}(\mathcal{F}(0,1-n))\otimes\mathcal{O}_{\bpn}(0,-1) \\
    \oplus\ho^{n}(\mathcal{F}(-1,3-n))\otimes\Omega^{n-2}(-1,n-2);    
\end{array}    
\end{equation}

\begin{equation}
\mathcal{C}^{0}=\begin{array}{c}
    \ho^{0}(\mathcal{F})\otimes\mathcal{O}_{\bpn} \quad \oplus \quad \ho^{1}(\mathcal{F}(0,-1))\otimes\Omega^{1}(0,1) \\
    \oplus\ho^{1}(\mathcal{F}(-1,1))\otimes\mathcal{O}_{\bp3}(-1,0) \quad \oplus \quad \ho^{n-1}(\mathcal{F}(0,1-n))\otimes\mathcal{O}_{\bpn}(0,-1) \\
    \oplus\ho^{n-1}(\mathcal{F}(-1,3-n))\otimes\Omega^{n-2}(-1,n-2) \quad \oplus \quad \ho^{n}(\mathcal{F}(-1,2-n))\otimes\mathcal{O}_{\bpn}(-1,-1); 
\end{array}   
\end{equation}

\begin{equation}
\mathcal{C}^{1}=\begin{array}{c}
    \ho^{0}(\mathcal{F}(0,-1))\otimes\Omega^{1}(0,1)\quad \oplus \quad \ho^{0}(\mathcal{F}(-1,1))\otimes\mathcal{O}_{\bpn}(-1,0) \\
    \oplus \ho^{1}(\mathcal{F}(0,-2))\otimes\Omega^{2}(0,-2) \quad \oplus \quad \ho^{1}(\mathcal{F}(-1,0))\otimes\Omega^{1}(-1,1) \\
    \oplus\ho^{n-1}(\mathcal{F}(-1,2-n))\otimes\mathcal{O}_{\bpn}(-1,-1) \quad \oplus \quad \ho^{n}(\mathcal{F}(0,1-n))\otimes\mathcal{O}_{\bpn}(0,-1) \\
    \ho^{n}(\mathcal{F}(-1,3-n))\otimes\Omega^{n-2}(-1,n-2).
\end{array}    
\end{equation}

\end{lemma}

\begin{proof}

The proof is an application of Theorem \cite[Theorem 8]{AO}.

Any sheaf $\mathcal{F}$ on $\bpn$ is the degree-$0$ cohomology of the complex from the statement of Theorem $8$ in \cite{AO}. If $\mathcal{C}^{-2}=0=\mathcal{C}^{2},$ then $\mathcal{F}$ is the cohomology of $$\mathcal{C}^{-1}\to\mathcal{C}^{0}\to\mathcal{C}^{1}.$$ 
In the notation of \cite[Theorem 8]{AO} for $p=2,$ the contributing terms come from $s=0,1,n-1$ and $n;$ for $s=0,$ we must satisfy $i=h+q=-2,$ but $h=0,1$ and $0\leq q\leq n-1.$ Hence, there are no terms contributing to $\mathcal{C}^{2}$ with $s=0.$ For $s=1,$ the integer $i$ must be equal to $-1,$ and, as above, there are no terms contributing to $\mathcal{C}^{2}$ with $s=1.$ When $s=n-1,$ then $i=n-3.$ This gives possibilities $h=0,$ $q=n-3,$ or $h=1,$ $q=n-4.$ Finally, if $s=n,$ the contributions of $h$ and $q$ come from $h=0$ and $q=n-2$ or $h=1$ and $q=n-3.$ Consequently, the terms that might contribute to $\mathcal{C}^{2}$ are $\ho^{n}(\mathcal{F}(0,2-n)),$ $\ho^{n}(\mathcal{F}(-1,4-n)),$ $\ho^{n-1}(\mathcal{F}(0,3-n))$ $\ho^{n-1}(\mathcal{F}(1,5-n)).$ The last term is included in this list only when $n\geq4.$ Nevertheless, these groups are zero by hypothesis. Thus $\mathcal{C}^{2}$ vanishes. In the same fashion, one proves that the rest of the vanishings in the hypothesis of the lemma lead to a trivial $\mathcal{C}^{-2}$ term. Finally, one checks that the following tables give the contributing terms to $\mathcal{C}^{-1},$ $\mathcal{C}^{0}$ and $\mathcal{C}^{1},$ and hence the monad $\mathbb{M}.$ 

for $p=-1$ the possibilities are 

$$\begin{array}{|c|c|}
\hline  s (i=h+q=s+1)   & (h,q) \\
\hline    0    & (0,1), (1,0) \\
\hline    1    & (0,2), (1,1) \\
\hline    n-1  & (1,n-1) \\
\hline    n    &  (0,n-1), (1,n-2) \\ \hline
\end{array}$$

for $p=0,$ we have 

$$\begin{array}{|c|c|}
\hline  s (i=h+q=s)   & (h,q) \\
\hline  0    & (0,0) \\
\hline  1    & (0,1), (1,0) \\
\hline  n-1  & (0,n-1),(1,n-2) \\
\hline   n   & (1,n-1) \\ \hline 
\end{array}$$

and for $p=1$ we also have

$$\begin{array}{|c|c|}
\hline  s (i=h+q=s-1)   & (h,q) \\
\hline  0    & \emptyset  \\
\hline  1    & (0,0) \\
\hline  n-1  & (0,n-2),(1,n-3) \\
\hline   n   & (0,n-1), (1,n-2) \\ \hline
\end{array}$$

\end{proof}

\section{The $\bp5$ case}\label{5dim}
In this section, we shall work on $\bp5$ as a prototype to build up examples. There is a full exceptional sequence of line bundles given by 
\begin{align}\label{excpional-seq-5}
\mathcal{O}_{\bp5}(-1&,-3),\mathcal{O}_{\bp5}(-1,-2),\mathcal{O}_{\bp5}(-1,-1),\mathcal{O}_{\bp5}(-1,0),\mathcal{O}_{\bp5}(-1,1),\notag \\ &\mathcal{O}_{\bp5}(0,-4),\mathcal{O}_{\bp5}(0,-3),\mathcal{O}_{\bp5}(0,-2),\mathcal{O}_{\bp5}(0,-1),\mathcal{O}_{\bp5}.
\end{align}

Consider the following codimension $2$ subvarieties: $\wp$ is the pull-back via the blow-down map $\pi:\bp5\to \p5,$ of a codimension $2$ linear space that does not contain the blown-up point, and $\kappa$ is a general codimension $1$ linear space contained in the exceptional divisor. The structure sheaf of the former has the resolution
\begin{equation}\label{reso-p}
    0\to\mathcal{O}_{\bp5}(-2,0)\to\mathcal{O}_{\bp5}(-1,0)^{\oplus 2}\to\mathcal{O}_{\bp5}\to\mathcal{O}_{\wp}\to0.
\end{equation}
Its second Chern class is $c_{2}(\mathcal{O}_{\wp})=-\xi^{2}$ in the Chow ring. The determinant of its normal bundle is $\Det(N_{\wp})\cong\mathcal{O}_{\p3}(2).$ The structure sheaf $\mathcal{O}_{\kappa}$ has a resolution, obtained by means of the moving lemma, which has the following form: 
\begin{equation}\label{reso-k}
    0\to\mathcal{O}_{\bp5}(-1,0)\to\begin{array}{c}\mathcal{O}_{\bp5}(-1,1)\\ \oplus \\ \mathcal{O}_{\bp5}(0,-1)\end{array}\to\mathcal{O}_{\bp5}\to\mathcal{O}_{\kappa}\to0.
\end{equation}
Its second Chern class is $c_{2}(\mathcal{O}_{\kappa})=\alpha^{2}-\xi^{2}$ and the determinant of its normal bundle is $\Det(N_{\kappa})\cong\mathcal{O}_{\p3}.$ As $\kappa$ is inside the exceptional divisor and $\wp$ does not contain the blown-up point, the two subvarieties have an empty intersection.

We set $X:=\wp\bigcup\kappa.$ Then the line bundle $\mathcal{S}:=\mathcal{O}_{\bp5}(2,0)$ on $\bp5$ restricts to $\Det(N_{X}),$ which can be seen by restricting $\mathcal{S}$ to each component. Moreover, one has the vanishings $\ho^{1}(\bp5, \mathcal{S}^{-1})=0$ and $\ho^{2}(\bp5, \mathcal{S}^{-1})=0,$ as it can be checked from \eqref{hi-lib}, or from the cohomology of line bundles on $\p5,$ as $\mathcal{S}$ is just the the pull-back of $\mathcal{O}_{\p5}(2)$ under the blow-down map. Hence, according to Theorem \ref{hart-ser}, there exists a locally free sheaf $\mathcal{F}$ fitting in the exact sequence
\begin{equation}\label{serre-bdle1}
    0\to\mathcal{O}_{\bp5}\to\mathcal{F}\to I_{X}\otimes\mathcal{S} \to0
\end{equation}
Furthermore, $\mathcal{F}$ is uniquely determined.
We set $\mathcal{E}:=\mathcal{F}\otimes\mathcal{O}_{\bp5}(-1,-1).$ Then

\begin{proposition}\label{properties}
The locally free sheaf $\mathcal{E}$ has the following cohomology table:{\tiny
$$\begin{array}{|c|c|c|c|c|c|c|c|c|c|c|}
\hline \mathcal{E}(-1,-3)&\mathcal{E}(-1,-2)&\mathcal{E}(-1,-1)&\mathcal{E}(-1,0)&\mathcal{E}(-1,1)&\mathcal{E}(0,-4)&\mathcal{E}(0,-3)&\mathcal{E}(0,-2)&\mathcal{E}(0,-1)&\mathcal{E}&\\
\hline  0&0&0&0&0&0&0&0&0&0& h^{5}\\
\hline 1&0&0&0&0&6&1&0&0&0& h^{4}\\
\hline 0&0&0&0&0&0&0&0&0&0& h^{3}\\
\hline  0&0&0&0&0&0&0&0&0&0& h^{2}\\
\hline 0&0&0&0&1&0&0&0&0&0& h^{1}\\
\hline  0&0&0&0&0&0&0&0&0&0&  h^{0}\\
\hline 
\end{array}$$
}
Moreover for every line bundle $\mathcal{R},$ we have $\ho^{i}(\bp5,\mathcal{E}\otimes\mathcal{R})=0$ for $i=2,3.$ 
\end{proposition}

\vspace{0.5cm}

\begin{proof}
We first need to compute the numbers  $h^{i}(\bp5,I_{X}(2,-2)\otimes\mathcal{M})$ for $\mathcal{M}$ ranging in the sequence \eqref{excpional-seq-5}. This can be achieved by using the restriction sequence 
\begin{equation}\label{rest-X}
    0\to I_{X}(1,-1)\to\mathcal{O}_{\bp5}(1,-1)\to\mathcal{O}_{X}\otimes\mathcal{O}_{\bp5}(1,-1)\to0,
\end{equation}
twisted by $\mathcal{M}.$ For instance, if $\mathcal{M}=\mathcal{O}_{\bp5},$ using the fact that the two components of $X$ are disjoint, we see that $$\mathcal{O}_{X}\otimes\mathcal{O}_{\bp5}(1,-1)=\mathcal{O}_{\kappa}(-1)\oplus\mathcal{O}_{\wp}\cong\mathcal{O}_{\p3}(-1)\oplus\mathcal{O}_{\p3}.$$ This gives 
$$h^{i}(\bp5,\mathcal{O}_{X})=\left\{\begin{array}{cc}
1     & \textnormal{if } i=0\\
0    & \textnormal{otherwise.}
\end{array}\right.$$
Furthermore, by \eqref{hi-lib} we have $$h^{i}(\bp5,\mathcal{O}_{\bp5}(1,-1))=\left\{\begin{array}{cc}
1     & \textnormal{if } i=0\\
0    & \textnormal{otherwise.}
\end{array}\right.$$
Since the direct image of the sheaf $\mathcal{O}_{X}\otimes\mathcal{O}_{\bp5}(1,-1)$ is $\mathcal{O}_{\wp},$ it then follows that $h^{0}(\bp5,I_{X}(1,-1))=0$ and hence $h^{i}(\bp5,I_{X}(1,-1))=0$ for all $i.$ Now by twisting the sequence \eqref{serre-bdle1} by $\mathcal{O}_{\bp5}(-1,-1)$ and using $\ho^{i}(\mathcal{O}_{\bp5}(-1,-1))=0,$ we obtain the result for $\mathcal{M}=\mathcal{O}_{\bp5}$ for all $i.$
The claim for the other values of $\mathcal{M}$ and for the other columns of the table follows from very similar computations.

Recall that, on an aCM variety of dimension $n,$ any line bundle has trivial cohomology $\ho^{i}$ for $1\leq i\leq n-1.$ In particular it is easy to check that $X$ is aCM. Hence, from any twist of \eqref{rest-X} by a line bundle $\mathcal{M}$ one has the following long exact sequence:
$$\underbrace{\ho^{1}(\mathcal{M}|_{X}\otimes\mathcal{O}_{\bp5}(1,-1)|_{X})}_{0}\to\ho^{2}(I_{X}\otimes\mathcal{M}\otimes\mathcal{O}_{\bp5}(1,-1))\to\underbrace{\ho^{2}(\mathcal{M}\otimes\mathcal{O}_{\bp5}(1,-1))}_{0}\to$$
$$\underbrace{\ho^{2}(\mathcal{M}|_{X}\otimes\mathcal{O}_{\bp5}(1,-1)|_{X})}_{0}\to\ho^{3}(I_{X}\otimes\mathcal{M}\otimes\mathcal{O}_{\bp5}(1,-1)\to\underbrace{\ho^{3}(\mathcal{M}\otimes\mathcal{O}_{\bp5}(1,-1))}_{0},$$
and we are done.

\end{proof}

Next we present some immediate consequences. 

\begin{lemma}\hspace{4cm}
\begin{itemize}
    \item[(i)] Let $H$ be the pull-back of a hyperplane in $\p5.$ Then $h^{i}(H,\mathcal{E}|_{H})=0$ for all $i=0,\cdots,4.$ 
    
    \item[(ii)] Let $E$ be the exceptional divisor in $\bp5.$ Then $h^{i}(E,\mathcal{E}|_{E})=\left\{\begin{array}{cc} 1 & \textnormal{ if } i=0 \\ 0 & \textnormal{ if } i\neq0 \end{array} \right.$.
\end{itemize}
Moreover, $\ho^{2}(H,\mathcal{E}|_{H}(k))=0$ and $\ho^{2}(E,\mathcal{E}|_{E}(k))=0$ for all $k\in\mathbb{Z}.$
\end{lemma}

\begin{proof}
The claims in both items can be derived from their respective restriction sequence $$0\to\mathcal{E}(-D)\to\mathcal{E}\to\mathcal{E}|_{D}\to0,$$ where the divisor $D$ is either $H$ or $E.$ Then, by using the cohomology table in Proposition \ref{properties}, one can easily obtain the result. Furthermore, by taking any twist of the restriction sequence, by any line bundle on $\bp5,$ the vanishing of $\ho^{2}(D,\mathcal{E}|_{D}(k))$ follows from Proposition \ref{properties}.

\end{proof}

\begin{corollary}\label{div-restriction}
Let the divisor $D$ be either $H$ or $E.$ Then the restriction $\mathcal{E}|_{D}$ of the bundle $\mathcal{E}$ is split. Furthermore, $\mathcal{E}|_{D}\cong\left\{ \begin{array}{cc}\mathcal{O}_{\p4}(-1)^{\oplus 2} & \textnormal{ if } D=H \\ \mathcal{O}_{\p4}\oplus\mathcal{O}_{\p4}(-2) & \textnormal{ if }D=E \end{array}\right.$ 
\end{corollary}

\begin{proof}
Using the main result in \cite{KPR}, the splitting is obtained from $\ho^{2}(D,\mathcal{E}|_{D}(k))=0.$ The degrees are obtained by Chern classes and the vanishing of cohomologies of line bundles on $\p4;$ for instance, when $D=H,$ we have $\ho^{i}(\mathcal{O}_{\p4}(a)\oplus\mathcal{O}_{\p4}(b))=0,$ then both $a$ an $b$ are in the range $-4\leq a, b\leq-1.$ Then, by computing the first Chern class of $\mathcal{E}|_{H}$ on $H=\p4,$ one gets $c_{1}(\mathcal{O}_{\p4}(a)\oplus\mathcal{O}_{\p4}(b))=-2h$ (here $h$ represents a hyperplane class in the Chow ring of $\p4$). This implies that $a=b=-1.$ A very similar computation yields the result for $D=E.$ 

\end{proof}

The result of the corollary also follows directly from the fact that of $\mathcal{E}|_{X}$ is the normal sheaf of $X$ in $\bp5.$


\begin{theorem}\label{instanton-proof}
The locally free sheaf $\mathcal{E}$ is an instanton bundle on $\bp5$ for the polarisation $\mathcal{L}=\mathcal{O}_{\bp5}(1,1)$ with first and second Chern classes given by $c_{1}(\mathcal{E})=-2\alpha,$ $c_{2}(\mathcal{E})=\xi^{2}$ respectively and of charge $c=8.$ 
\end{theorem}

\begin{proof}
The Chern class of $\mathcal{E}$ can be easily obtained from the exact triple \eqref{serre-bdle1}, twisted by $\mathcal{O}_{\bp5}(-1,1)$ and the resolutions of the structure sheaves for $\wp$ and $\kappa.$ 

By construction, $\mathcal{E}$ satisfies ${\rm(i)}$ of Definition \ref{Inst-def-bpn} for $\mathcal{L}=\mathcal{O}_{\bp5}(1,1).$ Furthermore, from the table in Proposition \ref{properties}, it is easy to see that $\mathcal{E}$ also satisfies conditions ${\rm (ii)-(iv)}.$ So it remains to check the $\mu_{\mathcal{L}}$-semi-stability.

We recall that $\mu_{\mathcal{L}}(\mathcal{E})=-\frac{30}{2}=-15.$ According to Corollary \ref{rk2-hoppe} we should show that $\ho^{0}(\bp5,\mathcal{E}(-p,-q))$ vanishes for all pairs $(p,q)$ satisfying 
\begin{equation}\label{range}
     q>-1-\frac{16}{15}p.
\end{equation}

For this aim, we use the inequalities
$$h^{0}(\bp5,\mathcal{E}(-p,-q))\leq h^{0}(\bp5,I_{X}(1-p,-q-1))$$
when $h^{0}(\bp5,\mathcal{O}_{\bp5}(1-p,-q-1))=0,$ and
\begin{equation}\label{dim-O} 
h^{0}(\bp5,I_{X}(1-p,-q-1))\leq h^{0}(\bp5,\mathcal{O}_{\bp5}(1-p,-q-1))=\sum_{k=0}^{1-p}\binom{3-q+k}{4},
\end{equation}
where the last term on the right is zero for $p>1.$ 
If $p=1,$ then $q$ is at least $-2,$ thus $3-q$ is at most $5.$ Then, one has: 
\begin{itemize}
    \item for $q$ positive or zero, the vanishing follows immediately; 
    \item for $q=-1$ the vanishing occurs from the fifth column (from left to right) of the cohomology table in Proposition \ref{properties};
    \item for $q=-2,$ take note of the fact that  $\ho^{0}(\bp5,\mathcal{O}_{\bp5}(-2,1))=0.$ Moreover, the map $I_{X}(0,1)\hookrightarrow I_{\kappa}(0,1)$ is injective, and the latter cohomology group has only one global section, since $\wp\cap\kappa=\emptyset.$ As a result $\ho^{0}(\bp5,I_{X}(0,1))=0,$ and hence $\ho^{0}(\bp5,\mathcal{E}(-1,2))=0.$ This vanishing can also be obtained by using the monad in the corollary below.
     
\end{itemize}
 
If $p=0,$ then $q>-1.$ For $q=0$ we have $\ho^{0}(\bp5,\mathcal{E})=0,$ by the cohomology table in Proposition \ref{properties}. The result holds for positive values of $q$ as well, because the map $\mathcal{O}_{\bp5}(0,-q)\to\mathcal{O}_{\bp5}$ is injective for the given values of $q.$

Now, consider negative values of $p.$ In this case, the highest value in the last term on the right of \eqref{dim-O} is $4-(p+q).$ If this is less than or equal to $3,$ then the desired vanishing occurs immediately, and one can easily check that this happens to be in the range $(p+q)\geq1.$ 
Moreover, according to \eqref{range}, we are considering the wider range $(p+q)>-1-\frac{p}{15}.$ As a result, we must ensure that $\ho^{0}(\bp5,\mathcal{E}(-p,-q))$ is trivial for $-15\leq p\leq-1.$ Remark that for each of these values of $p$ the only terms that will contribute are the ones for which $4-(p+q)=4,$ in other words, $q=-p.$ In this case $h^{0}(\bp5,\mathcal{O}_{\bp5}(1-p,p-1))=1.$ In addition, in all these cases, $h^{0}(\bp5,\mathcal{O}_{\bp5}(-1-p,p-1))=0.$ Hence,
$$\ho^{0}(\bp5,\mathcal{E}(-p,p))=\ho^{0}(\bp5,I_{X}(1-p,p-1)),$$ 
but the right hand cohomology group should be trivial because $\wp$ is not contained in the exceptional divisor. Thus $\mathcal{E}$ is semi-stable. 
\end{proof}

\begin{corollary}\label{monad-p5}
The instanton $\mathcal{E}$ is the cohomology of a monad of the form
\begin{equation}\label{monad}
    {\rm M}:\qquad0\to\mathcal{O}_{\bp5}(-1,-1)\to\begin{array}{c}\mathcal{O}_{\bp5}(-1,0)\\ \oplus  \\ \mathcal{O}_{\bp5}(0,-1)^{\oplus 6} \end{array}\to\Omega^{3}(0,3)\to0
\end{equation}
\end{corollary}

We recall that by $\Omega^{l}$ we mean $pr^{\ast}\Omega^{l}_{\p4},$ the pull-back via the projection $pr:\bp5\to\p4.$

\begin{proof}
The proof is an application of \cite[Theorem 8]{AO}, more precisely the complex \eqref{complex}; from the cohomology table in Proposition \ref{properties} one can easily see that the the only contributing terms to the complex comes from the values $s=1$ and $s=4.$ Moreover the values of $h$ can be only $0$ and $1,$ for the chosen exceptional collection. By a straightforward computation one obtains 
$$\mathcal{C}^{1}=\underbrace{\ho^{4}(\bp5,\mathcal{E}(0,-3))}_{\mathbb{C}}\otimes\Omega^{3}(0,3),$$

$$\mathcal{C}^{0}=\underbrace{\ho^{1}(\bp5,\mathcal{E}(-1,1))}_{\mathbb{C}}\otimes\mathcal{O}_{\bp5}(-1,0)\oplus\underbrace{\ho^{4}(\bp5,\mathcal{E}(0,-4))}_{\mathbb{C}^{\oplus 6}}\otimes\Omega^{4}(0,4),$$ where 
$\Omega^{4}(0,4)=\mathcal{O}_{\bp5}(0,-5)\otimes\mathcal{O}_{\bp5}(0,4)=\mathcal{O}_{\bp5}(0,-1),$

$$\mathcal{C}^{-1}=\underbrace{\ho^{4}(\bp5,\mathcal{E}(-1,-3))}_{\mathbb{C}}\otimes\Omega^{4}(0,4)$$ and $\mathcal{C}^{-2}=0.$
Since $\mathcal{E}$ is concentrated in degree $0,$ the result follows.   
\end{proof}

Let $\mathcal{M}^{ss}_{\bp5}(\mathcal{L},-2\alpha, \xi^{2})$ be the coarse moduli space of rank $2$ of $\mu_{L}$-semi-stable sheaves on $\bp5$ with $c_{1}=-2\alpha,$ $c_{2}=\xi^{2}.$

\begin{proposition}\label{component} \hspace{15cm}
There There exists an unobstructed component $\mathcal{I}_{\bp5}(\mathcal{L},-2\alpha, \xi^{2})\subset \mathcal{M}^{ss}_{\bp5}(\mathcal{L},-2\alpha,\xi^{2}),$ of dimension  $$5\leq{\rm dim}\hspace{0.1cm}\mathcal{I}_{\bp5}(\mathcal{L},-2\alpha, \xi^{2})\leq6,$$ whose points are instanton bundles.  
\end{proposition}

\begin{proof}

We recall that the instanton bundle $\mathcal{E}$ fits in the following short exact sequence: 
\begin{equation}\label{E-struc}
    0\to\mathcal{O}_{\bp5}(-1,-1)\to\mathcal{E}\to I_{X}(1,-1)\to0. 
\end{equation}
One can twist this sequence with $\mathcal{E}^{\vee}\simeq\mathcal{E}(0,2),$ then by using the fact that
$$h^{0}(\bp5,\mathcal{E}(-1,1))=\left\{\begin{array}{cc}
 1   & \textnormal{for } i=1 \\
 0   & \textnormal{otherwise}
\end{array}\right.$$
the long exact sequence of cohomology boils down to
\begin{align}\label{h-seq1}
    0\to&\ho^{0}(\bp5,\mathcal{E}\otimes\mathcal{E}^{\vee})\to\ho^{0}(\bp5,\mathcal{E}\otimes I_{X}(1,1))\to\mathbb{C}\to \notag\\
    &\ho^{1}(\bp5,\mathcal{E}\otimes\mathcal{E}^{\vee})\to\ho^{1}(\bp5,\mathcal{E}\otimes I_{X}(1,1))\to0,
\end{align}
together with the isomorphisms

\begin{equation}\label{iso1}
    \ho^{i}(\bp5,\mathcal{E}\otimes\mathcal{E}^{\vee})\simeq\ho^{i}(\bp5,\mathcal{E}\otimes I_{X}(1,1)) \textnormal{ for } i\geq2.
\end{equation}

On the other hand, if we twist \eqref{h-seq1} by $I_{X}(1,1),$ take the long exact sequence in cohomology and use the fact that
$$h^{0}(\bp5,I_{X})=\left\{\begin{array}{cc}
 1   & \textnormal{for } i=1 \\
 0   & \textnormal{otherwise,}
\end{array}\right.$$
then we obtain 
\begin{align}\label{h-seq2}
    0\to&\ho^{0}(\bp5,\mathcal{E}\otimes I_{X}(1,1))\to\ho^{0}(\bp5, I^{2}_{X}(2,0))\to\mathbb{C}\to\notag \\
    &\ho^{1}(\bp5,\mathcal{E}\otimes I_{X}(1,1))\to\ho^{1}(\bp5,I^{2}_{X}(2,0))\to0,
\end{align}
as well as 
\begin{equation}\label{iso2}
    \ho^{i}(\bp5,\mathcal{E}\otimes\mathcal{E}^{\vee})\simeq\ho^{i}(\bp5,\mathcal{E}\otimes I_{X}(1,1)) \textnormal{ for } i\geq2.
\end{equation}

Moreover, one has the exact sequence 
\begin{equation}\label{rest-X-twist}
    0\to I^{2}_{X}(2,0)\to I_{X}(2,0)\to\begin{array}{c}
    (\mathcal{O}_{\p3}(-1)\oplus \mathcal{O}_{\p3}(1)) \\
        \oplus   \\
    \mathcal{O}_{\p3}(1)^{\oplus2}     
    \end{array}\to0,
\end{equation}
obtained by twisting \eqref{rest-X} by $I_{X}(1,1)$ and using the equality $(I_{X}/I^{2}_{X})=N_{X}^{\vee}$ and the isomorphism $N_{X}^{\vee}\simeq N_{\kappa}^{\vee}\oplus N_{\wp}^{\vee}.$ The first summand in the latter decomposition gives the first row in the last entry of the sequence \eqref{rest-X-twist} and the second summand gives the lower row. Again, by taking the long exact sequence in cohomology, one gets
\begin{align}\label{h-seq3}
    0\to&\ho^{0}(\bp5,I^{2}_{X}(2,0))\to\ho^{0}(\bp5,I_{X}(2,0))\to\mathbb{C}^{\oplus12}\to \notag \\
    &\ho^{0}(\bp5,I^{2}_{X}(2,0))\to\ho^{0}(\bp5,I_{X}(2,0))\to0
\end{align}
and 
\begin{equation}
    \ho^{0}(\bp5,I^{2}_{X}(2,0))=0 \textnormal{ for } i\geq2.
\end{equation}
Combining the above vanishing and isomorphisms \eqref{iso1} and \eqref{iso2}, we conclude that  
$$\ho^{i}(\bp5,\mathcal{E}\otimes\mathcal{E}^{\vee})=0$$ for $i\geq2.$ In particular, we obtain the vanishing of the obstruction $\ho^{2}(\bp5,\mathcal{E}\otimes\mathcal{E}^{\vee}).$   

For a sheaf $\mathcal{F},$ we set $\delta^{0,1}(\mathcal{F}):=h^{0}(\mathcal{F})-h^{1}(\mathcal{F}).$ The long exact sequence of the triple $$0\to I_{X}(2,0)\to\mathcal{O}_{\bp5}(2,0)\to\mathcal{O}_{\kappa}\oplus\mathcal{O}_{\wp}(2)\to0$$
yields $\delta^{0,1}(I(2,0))=10.$ Then, from the long exact sequences \eqref{h-seq3}, we obtain $\delta^{0,1}(I^{2}_{X}(2,0))=-4.$ By combining this with the sequences \eqref{h-seq2} and \eqref{h-seq1}, we get $\delta^{0,1}(\mathcal{E}\otimes\mathcal{E}^{\vee})=-4.$ In particular, $h^{1}(\mathcal{E}\otimes\mathcal{E}^{\vee})$ is at least $5.$ 

\vspace{0.3cm}
{\bf \underline{Claim:}} $h^{0}(\mathcal{E}\otimes K(0,2))\leq1,$ where $K$ is the kernel of the second map in \eqref{monad}.
\vspace{0.3cm}

This follows from 
$$0\to\mathcal{E}\otimes K(0,2)\to\begin{array}{c}
\mathcal{E}(-1,2)  \\ \oplus \\
\mathcal{E}(0,1)^{\oplus6} 
\end{array}\to\mathcal{E}\otimes\Omega^{3}(0,5)\to.0$$
that 
\begin{equation}\label{leq1}
h^{0}(\mathcal{E}\otimes K(0,2))\leq h^{0}(\mathcal{E}(-1,2))+ 6\cdot h^{0}(\mathcal{E}(0,1).
\end{equation}

From the sequence
$$0\to\mathcal{O}_{\bp5}(-2,1)\to\mathcal{E}(-1,2)\to I_{X}(0,1)\to0$$ and the vanishing of $\ho^{0}(\mathcal{O}_{\bp5}(-2,1))=0,$
it follows that $h^{0}(\mathcal{E}(-1,2))\leq h^{0}(I_{X}(0,1)).$ But from the structure sequence \eqref{rest-X} one can easily show that $\ho^{0}(I_{X}(0,1))$ is trivial, and so is $\ho^{0}(\mathcal{E}(-1,2)).$ Furthermore, from 
$$0\to\mathcal{O}_{\bp5}(-1,0)\to\mathcal{E}(0,1)\to I_{X}(1,0)\to0$$ one can also show that 
$h^{0}(\mathcal{E}(-1,2))\leq h^{0}(I_{X}(1,0))=1.$ Then, by inequality \eqref{leq1}, the claim follows.

Finally, by twisting the sequence 
$$0\to\mathcal{O}_{\bp5}(-1,-1)\to K\to\mathcal{E}\to0$$ by $\mathcal{E}^{\vee}\simeq\mathcal{E}(0,2)$ and using the fact that $h^{0}(\bp5,\mathcal{E}(-1,1))=0$ and $h^{0}(\bp5,\mathcal{E}(-1,1))=1,$ it then follows that $h^{0}(\bp5,\mathcal{E}^{\vee}\otimes\mathcal{E})\leq\ho^{0}(\bp5,\mathcal{E}\otimes K(0,2))+1.$ Combining this with $\delta^{0,1}(\bp5,\mathcal{E}^{\vee}\otimes\mathcal{E})=-4$ yields ${\rm dim}\hspace{0.1cm}\mathcal{I}_{\bp5}(\mathcal{L},-2\alpha, -\xi^{2}+2\alpha^{2})|_{\mathcal{E}}=h^{1}(\bp5,\mathcal{E}^{\vee}\otimes\mathcal{E})=5$ or $6.$

\end{proof}

\begin{remark}

It is interesting to know whether one can find instantons of charges lower than $c=8$ and, in case this is not the minimal value, to know what the minimal one is. Another related question is that of knowing whether those instantons of minimal charges are Ulrich. 
We recall that, for a polarized variety $(X,\mathcal{L}),$ an  $\mathcal{L}$-{\em Ulrich bundle} is a bundle $\mathcal{E}$ satisfying 
\begin{itemize}
    \item $\ho^{i}(X,\mathcal{E}\otimes\mathcal{L}^{-i})=0,$ for $i>0,$ and
    \item $\ho^{j}(X,\mathcal{E}\otimes\mathcal{L}^{-(j+1)})=0,$ for $j<n.$
\end{itemize}
See \cite{beau} for a nice introduction, and references therein. In \cite[Proposition 6.2]{ccgm}, it was shown that a rank $2$ bundle $\mathcal{E}$ on $\bp3$ is an instanton of minimal charge if, and only if, its  twist $\mathcal{E}(1,1)$ is Ulrich. In our case, the rank two bundle $\mathcal{F}:=\mathcal{E}(1,1),$ given in the sequence \eqref{serre-bdle1}, is not Ulrich for the polarisation $\mathcal{L}=\mathcal{O}_{\bp5}(1,1).$ In fact, a simple calculation shows that the only thing preventing $\mathcal{F}$ from being an Ulrich bundle is $h^{5}(\bp5,\mathcal{F}(-5,-5))=54.$

\end{remark}

\section{The higher dimensional case}\label{higher-dim}

We now give a generalization of the construction in the $\bp5$ case by adopting the same Hartshorne-Serre correspondence strategy construction. 

We choose an odd integer $n\geq3.$ Then, as in Section \ref{5dim}, one can consider two codimension $2$ subvarieties: $\wp$ the pullback from $\pn$ of a linear codimension $2$ subspace not containing the blowup point, and $\kappa,$ a hyperplane of the exceptional divisor. The resolution of their structure sheaves are identical to \eqref{reso-p} and \eqref{reso-k}, respectively. The determinants of their normal bundles and their second Chern classes are given by 
\begin{equation}
    \begin{array}{cc}
    \det(N_{\wp})\cong\mathcal{O}_{\pnd}(2),     & c_{2}(\mathcal{O}_{\wp})=-\xi^{2} \\
     \det(N_{\kappa})\cong\mathcal{O}_{\pnd},    & c_{2}(\mathcal{O}_{\kappa})=\alpha^{2}-\xi^{2} 
    \end{array}
\end{equation}

We consider their union $X=\wp\bigcup\kappa.$ The line bundle $\mathcal{S}:=\mathcal{O}_{\bpn}(2,0)$ restricts to  $\det(N_{X})$ on $X.$ It is also easy to check that both cohomologies $\ho^{1}(\bpn,\mathcal{S})$ and $\ho^{2}(\bpn,\mathcal{S})$ are trivial. Then the conditions of Theorem \ref{hart-ser} are satisfied and there exists a rank $2$ bundle $\mathcal{F}$ given by 

\begin{equation}\label{F-seq}
    0\to\mathcal{O}_{\bpn}\to\mathcal{F}\to I_{X}(2,0)\to0,
\end{equation}
where $I_{X}$ is the ideal sheaf of $X.$ Moreover, one has the exact sequence:
\begin{equation}\label{IX-seq}
    0\to I_{X}(p,q)\to\mathcal{O}_{\bpn}(p,q)\to\mathcal{O}_{\pnd}(q)\oplus\mathcal{O}_{\pnd}(p+q)\to0,
\end{equation}
and it follows that $\ho^{i}(\bpn,I_{X}(p,q))=0$ $\forall p,q$ and $2\leq i\leq n-2.$ Consequently, we also have $\ho^{i}(\bpn,\mathcal{F}(p,q))=0$ $\forall p,q$ and $2\leq i\leq n-2.$

In this case, we choose the polarisation to be $\mathcal{L}:=\mathcal{O}(1,\frac{n-3}{2}).$ This makes the first Chern class of the instanton small: $c_{1}(\mathcal{E})=-2\alpha.$ We set $\mathcal{E}:=\mathcal{F}(-1,-1).$ Then, from \eqref{F-seq} twisted by $\mathcal{O}_{\bpn}(-1,-1),$ it follows that $\mathcal{E}$ fits in the exact sequence: 
\begin{equation}\label{E-seq}
    0\to\mathcal{O}_{\bpn}(-1,-1)\to\mathcal{E}\to I_{X}(1,-1)\to0,
\end{equation}

\begin{theorem}\label{Gen-const-Inst}
The locally free sheaf $\mathcal{E}$ is an instanton on $\bpn$ (for odd values $n\geq5$) with polarisation $\mathcal{L}=\mathcal{O}_{\bpn}(1,\frac{n-3}{2}),$ with first and second Chern classes given, respectively, by $c_{1}(\mathcal{E})=-2\alpha,$ $c_{2}(\mathcal{E})=\xi^{2}$ and of charge $c=\frac{(n-1)^{n-2}}{2^{n-2}}.$

\end{theorem}

\begin{proof}
The proof is a matter of checking the vanishings in Definition \ref{Inst-def-bpn}. The strategy is similar to the computations done in order to get the cohomology table in Proposition \ref{properties}. So we show the vanishing of one cohomology in each dimension, the rest of the computation is very similar.

\vspace{0.2cm}

\underline{\textbf{$\ho^{0}(\bpn,\mathcal{E}(0,-2))=0$}}: By setting $(p,q)=(1,-3),$ in \eqref{IX-seq}, and taking the long exact sequence in cohomology, we obtain the inequality $$h^{0}(\bpn,I_{X}(1,-3))\leq h^{0}(\bpn,\mathcal{O}_{\bpn}(1,-3))=0.$$ Hence $h^{0}(\bpn,I_{X}(1,-3))=0.$ The vanishing of the last term can be computed from \eqref{hi-lib}. On the other hand, if we twist \eqref{E-seq} by $\mathcal{O}_{\bpn}(0,-2),$ we get $$0=h^{0}(\bpn,\mathcal{O}_{\bpn}(-1,-3))=h^{0}(\bpn,\mathcal{E}(0,-2)),$$ since $h^{0}(\bpn,I_{X}(1,-3))=0.$

\vspace{0.2cm}

\underline{\textbf{$\ho^{1}(\bpn,\mathcal{E}(-1,-1))=0$}}: In the same way as above, if we set $(p,q)=(0,-2),$ in \eqref{IX-seq}, and take the long exact sequence in cohomology, we obtain the vanishings of both $h^{0}(\bpn,I_{X}(0,-2))$ and $h^{1}(\bpn,I_{X}(0,-2)).$ On the other hand, twisting \eqref{E-seq} by $h^{0}(\bpn,\mathcal{O}_{\bpn}(-1,-1))$ and using the fact that both $h^{0}(\bpn,\mathcal{O}_{\bpn}(-2,-2))$ and $h^{1}(\bpn,\mathcal{O}_{\bpn}(-2,-2))$ vanish, the claim follows.

\vspace{0.2cm}

\underline{\textbf{$\ho^{n-1}(\bpn,\mathcal{E}(0,3-n))=0$}}: In this case, Serre duality can be used, so $$\ho^{n-1}(\bpn,\mathcal{E}(0,3-n))\cong \ho^{1}(\bpn,\mathcal{E}(-2,0))^{\vee}.$$ This fits into the sequence \eqref{E-seq} twisted by $\mathcal{O}_{\bpn}(-2,0).$ Setting $(p,q)=(-1,-1)$ in the sequence \eqref{IX-seq}, we have $h^{i}(\bpn,I_{X}(-1,-1))=h^{i}(\bpn, \mathcal{O}_{\bpn}(-1,-1))=0$ for all $i.$ As a result $\ho^{1}(\bpn, \mathcal{E}(-2,0))\cong\ho^{1}(\bpn,\mathcal{O}_{\bpn}(-1,-3)).$ But the right hand side cohomology group in the last equality is zero, hence the claim is proved.

\vspace{0.2cm}

\underline{\textbf{$\ho^{n}(\bpn,\mathcal{E}(0,2-n))=0$}}: We have $$\ho^{n}(\bpn,I_{X}(1,1-n))\cong \ho^{n}(\bpn,\mathcal{O}_{\bpn}(-1,1-n))=0.$$ Hence $\ho^{n}(\bpn, \mathcal{E}(0,-2))\cong\ho^{n}(\bpn,\mathcal{O}_{\bpn}(-1,1-n)=0.$

\vspace{0.2cm}

Notice that the vanishing $\ho^{i}(\bpn,\mathcal{E}(p,q))=0$ $\forall p,q$ and $2\leq i\leq n-2$ follows easily as in the case of $\bp5.$

\vspace{0.2cm}

\underline{\textbf{Semi-Stability}}

We now check $\mu_{\mathcal{L}}$-semi-stability of $\mathcal{E},$ for the chosen polarisation $\mathcal{L}=\mathcal{O}_{\bpn}(1,\frac{n-3}{2}).$ The $\mathcal{L}$-slope of $\mathcal{E}$ is $$\mu_{\mathcal{L}}(\mathcal{E})=\frac{1}{2^{n-1}}\lbrack(n-3)^{n-1}-(n-1)^{n-1} \rbrack.$$ By using the generalized Hoppe criterion, in Corollary \ref{rk2-hoppe}, one can prove semi-stability by showing that $\ho^{0}(\bpn,\mathcal{E}(-p,-q))=0$ for all pairs $(p,q)$ such that 
\begin{equation}\label{q-n-range}
q>-1-\frac{p}{[1-(1-\frac{2}{n-1})^{(n-1)}]}.
\end{equation}
Following the same strategy as in the proof of Theorem \ref{instanton-proof}, one can use the inequalities
$$h^{0}(\bpn,\mathcal{E}(-p,-q))\leq h^{0}(\bpn,I_{X}(1-p,-1-q))$$
when $h^{0}(\bpn,\mathcal{O}_{\bpn}(-1-p,-1-q))=0$ and
\begin{equation}\label{dim-O-2} 
h^{0}(\bpn,I_{X}(1-p,-1-q))\leq h^{0}(\bpn,\mathcal{O}_{\bpn}(1-p,-1-q))=\sum_{k=0}^{1-p}\binom{n-2-q+k}{n-1},
\end{equation}
where the last term on the right is zero for $p>1.$ If $p=1$ then $q$ is at least $-2.$ This is because the quantity $\frac{1}{[1-(1-\frac{2}{n-1})^{(n-1)}]}$ is bounded above by the limit $$\frac{1}{1-e^{-2}}\approx 1,156517 \textnormal{ as }n\to\infty.$$
The difference between the upper bound and the lower bound of the binomial in \eqref{dim-O-2} is at most $1.$ The same argument as in the proof of Theorem \ref{instanton-proof} applies;
\begin{itemize}
    \item for $q$ positive or zero, the result follows immediately;
    \item for $q=-1,$ one can use the vanishing of $h^{0}(\bpn, I_{X})$ and $h^{0}(\bpn,\mathcal{O}_{\bpn}(-2,0));$
    \item for $q=-1,$ the vanishing of $h^{0}(\bpn, I_{X}(0,1))$ and $h^{0}(\bpn,\mathcal{O}_{\bpn}(-2,1))$ can be used instead.
\end{itemize}

For $p=0,$ $q$ is either positive or zero; we consider the inequality $h^{0}(\bpn,\mathcal{E})\leq h^{0}(\bpn,I_{X}(1,-1))=0,$ when $q=0.$ The result for positive $q$ follows from the inequalities $h^{0}(\bpn,\mathcal{E}(0,-q))\leq h^{0}(\bpn,\mathcal{E})=0.$

Finally, if $p$ is negative, then by looking at the highest term of the binomial upper bound in \eqref{dim-O-2}, we see that the vanishing occurs for $(p+q)\geq1.$ Furthermore, we have  $$(p+q)>-1-\frac{p\cdot (1-\frac{2}{n-1})^{n-1}}{[1-(1-\frac{2}{n-1})^{n-1}]}$$ according to the considered range \eqref{q-n-range}. This leaves us with a number of vanishings to check, specifically in the range $$\frac{[1-(1-\frac{2}{n-1})^{(n-1)}]}{(1-\frac{2}{n-1})^{(n-1)}}\leq p\leq-1.$$ However, because the function $f(n)=\frac{[1-(1-\frac{2}{n-1})^{(n-1)}]}{(1-\frac{2}{n-1})^{(n-1)}}$ tends to the value $\frac{e^{-2}}{1-e^{-2}}\approx 6,38905,$ as $n\to\infty,$ the number of cases to check is less than the $15$ cases we had for $\bp5.$ Finally, in every such case, the only value of $q$ for which the binomial in \eqref{dim-O-2} does not vanish is $q=-p.$ In those cases, one has $h^{0}(\mathcal{O}_{\bpn}(1-p,p-1))=1,$ since $p$ is negative. As a result, $h^{0}(\bpn,\mathcal{E}(-p,p))=h^{0}(\bpn,I_{X}(1-p,p-1)).$ This latter cohomology group is zero since $X$ is not contained in the exceptional divisor. Hence, $\mathcal{E}$ is $\mu_{\mathcal{L}}$-semi-stable.

\end{proof}

\begin{corollary}
The rank $2$ locally free sheaf $\mathcal{E}$ is the cohomology of a monad.
\end{corollary}

\begin{proof}
The claim follows easily from Lemma \ref{Gen-monad}. 
\end{proof}
The monad will take a simpler form than the one in Lemma \ref{Gen-monad}, since one can get rid of many more cohomology terms appearing in the complex $\mathbb{M}$ by using the Serre duality and more computations from the sequences \eqref{E-seq} and \eqref{IX-seq}, as in the proof of Theorem \ref{Gen-const-Inst}, above.

\subsection{Non locally free case}\label{non-loc-free}

We shall construct non-locally free examples of the instantons defined in \ref{Inst-def-bpn}. In the simplest case below we use the standard technique of elementary transformations.  
On $\bpn,$ we again choose a codimension $2$ subvariety $\wp\stackrel{\iota}{\hookrightarrow}\bpn$ which is the pullback of a linear space from $\pn$ under the monoidal transformation.
 
Let $\mathcal{M}$ denote the invertible sheaf $\mathcal{O}_{\wp}(-n-1)$ on $\wp.$ For a fixed instanton sheaf $\mathcal{E}$ on $\bpn,$ we suppose that there is a surjective map $\mathcal{E}\stackrel{j}{\twoheadrightarrow}\iota_{\ast}\mathcal{M}\otimes\omega^{-1}_{\bpn}.$ Then we have: 
\begin{proposition}
The torsion-free sheaf $\mathcal{G}:=\ker(\mathcal{E}\twoheadrightarrow\iota_{\ast}\mathcal{M}\otimes\omega^{-1}_{\bpn})$ is also an instanton sheaf on $\bpn.$
\end{proposition}

The sheaf $\mathcal{G}$ obtained in this way is clearly not locally free, and it is called an \emph{elementary transformation} of $\mathcal{E}$ by the triple $(\wp,\iota_{\ast}\mathcal{M},j).$

\begin{proof}
First, we remark that since $\mathcal{G}$ is a subsheaf of $\mathcal{E},$ with the same first Chern class and same $\mathcal{L}$-slope, then by the generalized Hoppe criteria, one can easily show that $\mathcal{G}$ is also $\mu_{\mathcal{L}}$-semi stable. 
The rest of the proof is completed by checking the vanishings in Definition \ref{Inst-def-bpn}. Twisting the short exact sequence 
$$0\to\mathcal{G}\to\mathcal{E}\to\iota_{\ast}\mathcal{O}_{\wp}\to0$$ by any line bundle $\mathcal{L}'$ and taking the associated long exact sequence in cohomology yields
$$\underbrace{\ho^{i-1}(\iota_{\ast}\mathcal{O}_{\wp}\otimes\mathcal{L}')}_{0}\to\ho^{i}(\mathcal{G}\otimes\mathcal{L}')\to\underbrace{\ho^{i}(\mathcal{E}\otimes\mathcal{L}')}_{0}$$ for all $2\leq i\leq n-2,$ because $\mathcal{E}$ is an instanton and $\wp$ is aCM. Moreover, as $\wp$ is of codimension two, if $i=n,$ one has $\ho^{n}(\mathcal{G}\otimes\mathcal{L}')=\ho^{n}(\mathcal{E}\otimes\mathcal{L}').$ Hence, the conditions in item ${\rm (iii)}$ are automatically satisfied.

We notice that any line bundle $\mathcal{O}_{\bpn}(p,q)$ restricts to $\mathcal{O}_{\wp}(p+q).$ Let $\mathcal{L}'$ be one of the following line bundles
\begin{itemize}
    \item  $\mathcal{O}_{\bpn}(0,-2)$ or  $\mathcal{O}_{\bpn}(-1,-0)$ for $i=0;$
    \item $\mathcal{O}_{\bpn}(-1,-1)$ or  $\mathcal{O}_{\bpn}(0,-3)$ for $i=1;$
    \item $\mathcal{O}_{\bpn}(0,3-n)$ or  $\mathcal{O}_{\bpn}(-1,5-n)$ for $i=0$
\end{itemize}
We shall do a similar computation as for the vanishing of the middle range $2\leq i\leq n-2.$ Furthermore, the instanton conditions for $\mathcal{E}$ can be used as follows. In the first case, one has $$\ho^{i}(\mathcal{G}\otimes\mathcal{L}')\hookrightarrow\underbrace{\ho^{i}(\mathcal{E}\otimes\mathcal{L}')}_{0}$$ for both $\mathcal{L}'=\mathcal{O}_{\bpn}(0,-2)$ or $\mathcal{O}_{\bpn}(-1,-0).$

In the second case one has 

$$\underbrace{\ho^{0}(\iota_{\ast}\mathcal{O}_{\wp}(-3))}_{0}\to\ho^{1}(\mathcal{G}(0,-3))\to\underbrace{\ho^{1}(\mathcal{E}(0,-3))}_{0}$$
or

$$\underbrace{\ho^{0}(\iota_{\ast}\mathcal{O}_{\wp}(-2)\otimes\mathcal{L}')}_{0}\to\ho^{1}(\mathcal{G}(-1,-1))\to\underbrace{\ho^{1}(\mathcal{E}(-1,-1))}_{0},$$

and in the last case one has

$$\underbrace{\ho^{n-2}(\iota_{\ast}\mathcal{O}_{\wp}(3-n))}_{0}\to\ho^{n-1}(\mathcal{G}(0,3-n))\to\underbrace{\ho^{n-1}(\mathcal{E}(0,3-n))}_{0}$$
or

$$\underbrace{\ho^{n-2}(\iota_{\ast}\mathcal{O}_{\wp}(4-n))}_{0}\to\ho^{n-1}(\mathcal{G}(-1,5-n))\to\underbrace{\ho^{n-1}(\mathcal{E}(-1,5-n))}_{0}.$$

\end{proof}

The study of the smoothability of the torsion free instanton sheaves $\mathcal{G},$ obtained via elementary transformations, can be useful, for instance, to produce instanton moduli components on $\bpn$ with higher charges.


\section{Instantons on $\bpn$ for even $n$}\label{even}

To give an idea, let'us look at the case of $\bp4.$ In this case, one can repeat, {\em mutatis mutandis}, the same construction as in the odd case. However, $\mu_{\mathcal{L}}$-semi-stability might fail, as shown in the example below. This is one of the cases we explored, but many other examples seem to fail the definition because of semi-stability. This could be an artifact of the particular construction in use. But one is tempted to relax the semi-stability and call {\em instantons} all sheaves satisfying the cohomological conditions in Definition \ref{Inst-def-bpn} in order to avoid this problem. In this case, some of the instanton bundles will not be semi-stable for the chosen polarisation, and one may seek another one in which the same instanton bundles are semi-stable. Remark that, on $\pn,$ the Definition \ref{classic-inst} does not include semi stability, but as shown in \cite[Theorem 15]{J1}, semi-stability follows for low values of the rank.

We consider the following codimension $2$ subvarieties: $\mathcal{Q}_{1}$ is the pull-back via the blow-down map $\pi:\bp4\to \p4,$ of a smooth codimension $2$ quadric surface that does not contain the blown up point, and $\mathcal{Q}_{2}$ is a smooth codimension $1$ quadric surface contained in the exceptional divisor. The structure sheaf of the former has a resolution
\begin{equation}\label{reso-p-4}
    0\to\mathcal{O}_{\bp4}(-3,0)\to\mathcal{O}_{\bp4}(-2,0)\oplus \mathcal{O}_{\bp4}(-1,0)\to\mathcal{O}_{\bp4}\to\mathcal{O}_{\mathcal{Q}_{1}}\to0
\end{equation}
Its second Chern class is $c_{2}(\mathcal{O}_{\mathcal{Q}_{1}})=-2\xi^{2},$ and the determinant bundle of the normal sheaf to $\mathcal{Q}_{1}$ is $\Det(N_{\mathcal{Q}_{1}})=\mathcal{O}_{\mathcal{Q}_{1}}(3).$

The structure sheaf $\mathcal{O}_{\mathcal{Q}_{2}}$ has a resolution
\begin{equation}\label{reso-k-4}
    0\to\mathcal{O}_{\bp4}(-3,-1)\to\begin{array}{c}\mathcal{O}_{\bp4}(-1,1)\\ \oplus \\ \mathcal{O}_{\bp4}(-2,-2)\end{array}\to\mathcal{O}_{\bp4}\to\mathcal{O}_{\mathcal{Q}_{2}}\to0.
\end{equation}
Its class is given by $c_{2}(\mathcal{O}_{\mathcal{Q}_{2}})=2[\alpha^{2}-\xi^{2}]$ and the determinant of its normal bundle is $\Det(N_{\mathcal{Q}_{2}})=\mathcal{O}_{\bp4}(3,1)|_{\mathcal{Q}_{2}}\cong\mathcal{O}_{\mathcal{Q}_{2}}(1).$

Let $X$ be the disjoint union of $\mathcal{Q}_{1}$ and $\mathcal{Q}_{2},$ and set $\mathcal{S}:=\mathcal{O}_{\bp4}(2,1)$ on $\bp4.$ Then $\mathcal{S}$ restricts to $\Det(N_{X}).$ Moreover, both vector spaces $\ho^{1}(\bp4, \mathcal{S}^{-1})$ and $\ho^{2}(\bp4, \mathcal{S}^{-1})$ are trivial. Hence, according to Theorem \ref{hart-ser}, there exists a uniquely determined locally free sheaf $\mathcal{F}$ fitting in the exact sequence
\begin{equation}\label{serre-bdle-even}
    0\to\mathcal{O}_{\bp4}\to\mathcal{F}\to I_{X}(2,1) \to0
\end{equation}
We set $\mathcal{E}:=\mathcal{F}\otimes\mathcal{O}_{\bp4}(-1,-1).$

\begin{proposition}
$\mathcal{E}$ has first Chern class $c_{1}(\mathcal{E})=-\alpha,$ second Chern class $c_{2}(\mathcal{E})=2(\xi^{2}-\alpha^{2}),$ and charge $c=2.$ Moreover, it satisfies
\begin{itemize}
    \item[(i)] $\ho^{0}(\bp4,\mathcal{E}(0,-2))=0,$ $\ho^{0}(\bp4,\mathcal{E}(-1,0))=0;$
    \item[(ii)] $\ho^{1}(\bp4,\mathcal{E}(-1,-1))=0,$ $\ho^{1}(\bp4,\mathcal{E}(0,-3))=0;$  
    \item[(iii)] $\ho^{3}(\bp4,\mathcal{E})=0,$ $\ho^{3}(\bp4,\mathcal{E}(-1,1))=0$
    \item[(iv)] $\ho^{4}(\bp4,\mathcal{E}(0,-2))=0,$ $\ho^{4}(\bp4,\mathcal{E}(-1,0))=0$ and $\ho^{2}(\bp4,\mathcal{E}\otimes\mathcal{L}')=0$, for any invertible sheaves $\mathcal{L}'$ on $\bp4$
\end{itemize}

\end{proposition}

\begin{proof}
The proof is similar to the one given for the cohomology table in Proposition \ref{properties}. 
\end{proof}

In particular $\mathcal{E}$ is the cohomology of a monad. For $\mu_{L}$-semi-stability, one can adopt the same reasoning as in the proof of Theorem \ref{instanton-proof}, that is, by using the generalized Hoppe criteria again. In the present case, one considers pairs of integers $(p,q)$ such that 
\begin{equation}\label{range-even}
    q>-\frac{1}{2}-\frac{8}{7}p
\end{equation}

For $p\geq0$ and $q$ as in the inequality above, we still obtain the vanishing of $\ho^{0}(\bp4,\mathcal{E}(-p,-q)).$ However, for the negative $p,$ the rest of the vanishings fail: For instance, if $p=1,$ then $q$ is at least $1.$ If we consider the smallest value $q=1$, then one has $\ho^{0}(\bp4,\mathcal{E}(1,-1))=\ho^{0}(\bp4,I_{X}(2,-1)).$ The exact sequence 
$$0\to\ho^{0}(\bp4,I_{X}(2,-1))\to\underbrace{\ho^{0}(\bp4,\mathcal{O}_{\bp4}(2,-1))}_{\mathbb{C}^{5}}\to\underbrace{\ho^{0}(\bp4,\mathcal{O}_{\mathcal{Q}_{1}}(1)\oplus\mathcal{O}_{\mathcal{Q}_{2}}(-1))}_{\mathbb{C}^{4}}$$ $$\to\ho^{1}(\bp4,I_{X}(2,-1))\to0,$$ yields $h^{0}(\bp4,I_{X}(2,-1))= h^{1}(\bp4,I_{X}(2,-1))+1>0.$ Thus, $\mathcal{E}$ fails to be semi-stable.

Other examples have been investigated on $\bp4$ and they give a similar structure. This issue will be investigated in future work in order to produce analogous instantons in the even-dimensional case.

\vspace{0.2cm}

{\small

Abdelmoubine Amar Henni   

Departamento de Matem\'atica MTM - UFSC

Campus Universit\'ario Trindade CEP 88.040-900 Florian\'opolis-SC, Brazil

e-mail: henni.amar@ufsc.br   

}


\vspace{0.5cm}



\begin{thebibliography}{99}





\bibitem{AO} Ancona, V., Ottaviani, G; 
Canonical resolutions of Schubert and Brieskorn varieties. 
Complex analysis (Wuppertal, 1991) K Diederich ed., Aspects Math., {\bf E 17}, Friedr. Vieweg, Braunschweig, (1991), 14-19.

\bibitem{AMM} Andrade, A., Marchesi, S, Mir\'o-Roig, Rosa M.;
Irreducibility of the moduli space of orthogonal instanton bundles on $\pn.$ 
Rev. Mat. Complut. {\bf 33} (2020), no. 1, 271-294.

\bibitem{AM} Antonelli, V., Malaspina, F.;
Instanton bundles on the Segre threefold with Picard number three.
arXiv:1909.10895. [math.AG]

\bibitem{AM1} Antonelli, V., Malaspina, F.;
H-Instanton bundles on three-dimensional polarized projective varieties.
arXiv:2007.04164. [math.AG]

\bibitem{arrondo} Arrondo, E.;
A Home-Made Hartshorne-Serre Correspondence
Revista Matem\'atica Complutense, Vol {\bf 20}, 2 (2007).

\bibitem{atiyah} Atiyah, M. F.; 
Geometry of Yang-Mills Fields, 
Publication of Scuola Normale Superiore, Pisa (1979).


\bibitem{ADHM}
Atiyah, M., Drinfeld, V., Hitchin, N., Manin, Yu.;
Construction of instantons.
Phys. Lett. {\bf 65A}, 185-187 (1978) 

\bibitem{rava} Bartocci, C., Bruzzo, U., Rava, Claudio L. S.;
Monads for framed sheaves on Hirzebruch surfaces.
Advances in Geometry, vol. {\bf 15}, no. 1, 2015, pp. 55-76.

\bibitem{beau} Beauville, A.; 
An introduction to Ulrich bundles. 
European Journal of Mathematics 4, 26-36 (2018).

\bibitem{BPST} Belavin, A. A., Polyakov, A. M., Schwarz, A. S, Tyupkin, Yu. S.;
Pseudoparticle solutions of the Yang-Mills equations.
Phys. Letters {\bf B}, {\bf 59}: 85-87.


\bibitem{bott} Bott, R.; 
Homogeneous vector bundles. 
Ann. of Math. 66, 203-248 (1957).


\bibitem{ccgm} Casnati, C., Coskun, E., Genk, O., Malaspina, F.;
Instanton bundles on the blow up of the projective $3$-space at a point.
arXiv:1909.10281[math.AG].

\bibitem{gaia} Comaschi, G.;
Stable linear systems of skew-symmetric forms of generic rank less than or equal to $4.$
arXiv:2003.14201 [math. AG].


\bibitem{costa-miro} Costa, L., Mir\`o-Roig, R.M.;
Instanton bundles vs Ulrich bundles on projective spaces.
Beitr Algebra Geom {\bf 62}, 429-439 (2021).

\bibitem{faenzi} Faenzi. D.;
Even and odd instanton bundles on Fano threefolds of Picard number $1.$
Manuscripta Math. {\bf 144} (2014), 199--239.

\bibitem{floystad} 
G. Fl\o ystad,
Monads on projective spaces,
Comm. Algebra {\bf 28} (2000), 5503--5516.

\bibitem{FJM} Franco, E., Jardim, M., Marchesi, S;
Branes in the moduli space of framed sheaves. Bull. Sci. Math. {\bf 141} (2017), no. 4, 353-383.

\bibitem{fulton} Fulton, W.;
Introduction to Toric Varieties. 
Princeton Univ. Press, NJ, (1993)

\bibitem{Hart1}
Hartshorne, R.;
Stable Reflexive sheaves,
Math. Ann. {\bf 254} (1980), 121--176.

\bibitem{HH} Hartshorne, R., Hirschowitz, A.; 
Cohomology of a general instanton bundle. 
Ann. Sci. Ecole Norm. Sup. (4) 15, 365--390 (1982).

\bibitem{Henni0} Henni, A. Amar:
Monads for torsion-free instanton sheaves on multi-blow-ups of the projective plane.
Int. J. of Mathematics Vol. {\bf 25}, No. 1 (2014) 1450008.

\bibitem{Henni1} Henni, A. Amar:
Torsion free instanton sheaves on the blow-up of $\p3$ at a point. 
arXiv:2005.11291 [math. AG].

\bibitem{HJ} Henni A. A., Jardim, M.;
Monad constructions of omalous bundles
J. Geom. and Physics {\bf 74} (2013) 36-42.

\bibitem{HJM} Henni, A. A., Jardim, M., Martins, R. V.;
ADHM construction of perverse instanton sheaves
Glasgow Math. J. {\bf 57} (2014) 285-321.

\bibitem{Huy} Huybrechts, D., Lehn, M.; 
The geometry of moduli spaces. 
Aspects. Math, {\bf E 31}, Friedr. Vieweg $\&$ Sohn, Braunschweig, 1997.

\bibitem{Isk1}
Iskovskikh, V., A.;
Fano 3-Folds. I. 
Izv. Akad. Nauk SSSR - Ser. Mat. Tom {\rm 41} (1977), no. 3, 485--527.

\bibitem{J1} Jardim M.; 
Instanton sheaves on complex projective spaces, 
Collect. Math. 57 (2006) 69-91.


\bibitem{JMPS} Jardim, M., Menet, G., Prata, D., S\`a Earp, H.; 
Holomorphic bundles for higher dimensional gauge theory,
Bull. London Math. Soc., {\bf 49}, 117-132 (2017).

\bibitem{KPR} Kumar, N.M., Peterson, C., Rao,  A.P.;
Monads on projective spaces, 
Manuscripta Math. {\bf 112} (2003), 183-189.

\bibitem{kuznetsov} Kuznetsov. A.;
Instanton bundles on Fano threefolds.
Cent. Eur. J. Math. {\bf 10} (2012),1198-1231.


\bibitem{MMJ} Malaspina, F., Marchesi, S., Pons-Llopis, J.;
Instanton bundles on the Flag variety $F(0,1,2).$
arXiv:1706.06353 [math.AG].


\bibitem{MT} Maruyama, M., Trautmann, G.; 
Limits of instantons. 
Internat. J. Math {\bf 3} (1992), 213-276.


\bibitem{OSS}
Okonek, C., Schneider, M., Spindler, H.;
Vector bundles on complex projective spaces.
Progress in mathematics 3, Birkhauser, Boston, 1980.


\bibitem{serre} Serre, J.P.;
Sur les modules projectifs, 
S\'eminaire Dubreil-Pisot (1960/61), Secr. Math.Paris, expos\'e 2 (1961).


\bibitem{Sanna} Sanna,  G.;
Rational curves and instantons on the Fano threefold $Y_{2}.$
arXiv:1411.7994 [math.AG] (PhD thesis).


\end{thebibliography}
\end{document}